\documentclass{amsart}
\usepackage{amssymb,latexsym}
\theoremstyle{plain}
\newtheorem{theorem}{Theorem}

\newtheorem{lemma}{Lemma}
\theoremstyle{definition}
\newtheorem{definition}{Definition}

\newtheorem{notation}{Notation}
\newtheorem{remark}{Remark}
\newtheorem{question}{Question}
\date{}

\begin{document}

\title[Weak defectivity]
{Symmetric tensor rank with a tangent vector: a generic uniqueness theorem}
\author{Edoardo Ballico, Alessandra Bernardi}
\address{Dept. of Mathematics\\
  University of Trento\\
38123 Povo (TN), Italy}
\address{GALAAD, INRIA M\'editerran\'ee,  BP 93, 06902 Sophia 
Antipolis, France.}
\email{ballico@science.unitn.it, alessandra.bernardi@inria.fr}
\thanks{The authors were partially supported by CIRM of FBK Trento 
(Italy), Project Galaad of INRIA Sophia Antipolis M\'editerran\'ee 
(France), Institut Mittag-Leffler (Sweden), Marie Curie: Promoting science (FP7-PEOPLE-2009-IEF), MIUR and GNSAGA of 
INdAM (Italy).}
\subjclass{14N05; 14M17}
\keywords{Veronese variety; tangential variety; join; weak defectivity}

\begin{abstract}
Let $X_{m,d}\subset \mathbb {P}^N$, $N:= \binom{m+d}{m}-1$, be the order  $d$ Veronese
embedding of $\mathbb {P}^m$. Let $\tau (X_{m,d})\subset \mathbb {P}^N$, be the tangent developable of $X_{m,d}$. For
each integer $t \ge 2$ let $\tau (X_{m,d},t)\subseteq \mathbb {P}^N$, be the join of $\tau (X_{m,d})$ and $t-2$ copies
of $X_{m,d}$. Here we prove that if $m\ge 2$, $d\ge 7$ and $t \le 1 + \lfloor \binom{m+d-2}{m}/(m+1)\rfloor$, then for a general
$P\in \tau (X_{m,d},t)$ there are uniquely determined $P_1,\dots ,P_{t-2}\in X_{m,d}$ and a unique tangent vector $\nu$ of $X_{m,d}$ such
that $P$ is in the linear span of $\nu \cup \{P_1,\dots ,P_{t-2}\}$, i.e. a degree $d$ linear form $f$ (a symmetric tensor $T$ of  order $d$) associated to $P$ may be written as
$$f = L_{t-1}^{d-1}L_t + \sum _{i=1}^{t-2} L_i^d, \; \; \; \; (T = v_{t-1}^{\otimes (d-1)}v_t + \sum _{i=1}^{t-2} v_i^{\otimes d})$$
with $L_i$ linear forms on $\mathbb {P}^m$ ($v_i$ vectors over a vector field  of dimension $m+1$ respectively), $1 \le i \le t$, that are uniquely determined (up to a constant).
\end{abstract}

\maketitle

\section{Introduction} 
In this paper we want to address the question of the uniqueness of a particular decomposition for certain given homogeneous polynomials. 
 An analogous question can be rephrased in terms of uniqueness of a particular tensor decomposition of certain given symmetric tensors. In fact, given a homogeneous polynomial $f$ of degree $d$ in $m+1$ variables defined over an algebraically closed field $\mathbb{K}$, there is an obvious way to associate  a symmetric tensor $T \in S^d(V_{\mathbb{K}})$, with $\dim(V_{\mathbb{K}})=m+1$, to  the form $f$.
We will always work over an algebraically closed field $\mathbb{K}$ such that $\mbox{char}(\mathbb{K})=0$.
Fix integers $m\ge 2$ and $d\ge 3$.
Let $j_{m,d}: \mathbb {P}^m \hookrightarrow \mathbb {P}^N$, $N:= \binom{m+d}{m}-1$, be the order $d$ Veronese embedding of $\mathbb {P}^m$ and set $X_{m,d}:= j_{m,d}(\mathbb {P}^m)$ (we often write $X$ instead of $X_{m,d}$). Let $\mathbb{K}[x_0, \ldots , x_m]_d$ be the polynomial ring of homogeneous degree $d$ polynomials in $m+1$ variables over $\mathbb{K}$ and let $V_{\mathbb{K}}^*$ be the dual space of $V_{\mathbb{K}}$. Since obviously $\mathbb{P}^m\simeq \mathbb{P}(\mathbb{K}[x_0, \ldots , x_m]_1)\simeq \mathbb{P}(V_{\mathbb{K}}^*)$, an element of the Veronese variety $X_{m,d}$ can be interpreted either as the projective class of a $d$-th power of a linear form $L\in \mathbb{K}[x_0, \ldots , x_m]_1$ or as the projective class of a symmetric tensor $T\in S^d(V_{\mathbb{K}}^*)\subset (V_{\mathbb{K}}^{*})^{\otimes d}$ for which there exists $v\in V_{\mathbb{K}}^*$ s.t. $T=v^{\otimes d}$.
\\
For each integer $t$ such that $1\le t \le N$ let $\sigma _t(X)$ denote the closure in $\mathbb {P}^N$ of the union of all $(t-1)$-dimensional linear subspaces spanned by $t$ points of $X$ (the $t$-secant variety of $X$). From this definition one can understand that the generic element of $\sigma_t(X_{m,d})$ can be interpreted either as $[f]=[L_1^d + \cdots + L_t^d]\in \mathbb{P(K}[x_0, \ldots , x_m]_d)$ with $L_1, \ldots , L_t \in \mathbb{K}[x_0 , \ldots , x_m]_1$ or as $[T]=[v_1^{\otimes d}+ \cdots + v_t^{\otimes d}]\subset \mathbb{P}(S^d(V_{\mathbb{K}}^*))$ with $v_1 , \ldots , v_t\in V^*_{\mathbb{K}}$. For a given form $f$ (or a symmetric tensor $T$), the minimum integer $t$ for which there exists such a decomposition is called the symmetric rank of $f$ (or of $T$). Finding those  $v_i$'s, $i=1, \ldots , t$ such that $T=v_1^{\otimes d}+ \cdots + v_t^{\otimes d}$, with $t$ the symmetric rank of $T$,  is known as the Tensor Decomposition problem  and it is a generalization of the Singular Value Decomposition problem for symmetric matrices (i.e. if $T\in S^2(V^{*}_{\mathbb{K}})$).  The existence and the possible uniqueness of the decompositions of a form $f$ as $L_1^{d}+ \cdots + L_t^d$ with $t$ minimal is studied in certain cases in \cite{cc}, \cite{cr}, \cite{m1}, \cite{m2}.
\\
Let $\tau (X)\subseteq \mathbb{P}^N$ be the tangent developable of
$X$, i.e. the closure in $\mathbb{P}^N$ of the union of all embedded tangent spaces $T_PX$, $P\in X$.
Obviously
$\tau (X) \subseteq \sigma _2(X)$ and $\tau (X)$ is integral. 
Since $d\ge 3$, the variety $\tau (X)$ is a divisor of $\sigma _2(X)$ (\cite{cgg}, Proposition 3.2).  An element in $\tau(X_{m,d})$ can be described  both as  $[f]\in \mathbb{P}(\mathbb{K}[x_0, \ldots , x_m]_d)$ for which there exists two linear forms $L_1, L_2\in \mathbb{K}[x_0, \ldots , x_m]_1$ such that $f=L_1^{d-1}L_2$, and as $[T]\in \mathbb{P}(\in S^d(V_{\mathbb{K}}^*))$ for which there exists two vectors $v_1,v_2\in V^*_{\mathbb{K}}$ such that $T=v_1^{\otimes d-1}v_2$ (\cite{cgg}, \cite{bgi}).
\\
Fix integral positive-dimensional subvarieties $A_1,\dots ,A_s \subset \mathbb {P}^N$, $s\ge 2$.
The join $[A_1,A_2]$ is the closure in $\mathbb {P}^N$ of the union of all lines spanned by a point of $A_1$ and a different point of $A_2$. If $s\ge 3$ define inductively the join $[A_1,\dots ,A_s]$ by the formula $[A_1,\dots ,A_s]:= [[A_1,\dots ,A_{s-1}],A_s]$. 
The join $[A_1,\dots ,A_s]$ is an integral variety and $\dim ([A_1,\dots ,A_s])\le \min \{N,s-1 +\sum _{i=1}^{s} \dim (A_i)\}$. 
The integer $\min \{N,s-1 +\sum _{i=1}^{s} \dim
(A_i)\}$ is called the {\it expected dimension} of the join $[A_1,\dots ,A_s]$. Obviously $[A_1,\dots ,A_s]
=[A_{\sigma (1)},\dots ,A_{\sigma (s)}]$ for any permutation $\sigma : \{1,\dots ,s\} \to \{1,\dots ,s\}$. The secant variety $\sigma _t(X)$, $t\ge 2$, is the join of $t$ copies of $X$.
For each integer
$t\ge 3$ let
$\tau (X,t)
\subseteq
\mathbb{P}^N$ be the join of
$\tau (X)$ and
$t-2$ copies of $X$.
We recall that $\min \{N,t(m+1)-2\}$ is the expected dimension of $\tau (X,t)$, while $\min \{N,t(m+1)-1\}$ is the expected dimension of $\sigma _t(X)$. In the range of triples $(m,d, t)$ we will meet in this paper
both $\tau (X,t)$ and $\sigma _t(X)$ have the expected dimensions and hence $\tau (X,t)$ is a divisor of $\sigma _t(X)$. An element in $\tau(X_{m,d},t)$ can be described both as $[f]\in \mathbb{P}(\mathbb{K}[x_0, \ldots , x_m]_d)$ for which there exist linear forms $L_1 \ldots , L_t \in \mathbb{K}[x_0, \ldots , x_m]_1$ such that $f=L_{t-1}^{d-1}L_t+ \sum_{i=1}^{t-2}L_i^d$, and as $[T]\in \mathbb{P}(S^d(V^*_{\mathbb{K}}))$ for which there exist $v_1, \ldots , v_t\in V^*_{K}$ such that $T=v_{t-1}^{\otimes (d-1)}v_t+ \sum_{i=1}^{t-2}v_i^{\otimes d}$.

After \cite{bb}, it is natural to ask the following question.

\begin{question}\label{w3}
Assume $d\ge 3$ and $\tau (X,t) \ne \mathbb{P}^N$. Is a general point of $\tau (X,t)$
in the linear span of a unique set $\{P_0,P_1,\dots ,P_{t-2}\}$ with $(P_0,P_1,\dots ,P_{t-2})\in
\tau (X)\times X^{t-2}$?
\end{question}

For non weakly $(t-1)$-degenerate subvarieties of $\mathbb{P}^N$ the corresponding question is true by \cite{cr}, Proposition 1.5. Here we answer it for a large set of triples of integers $(m ,d, t)$ and prove the following result.

\begin{theorem}\label{i1}
Fix integers $m \ge 2$ and $d\ge 6$. If $m\le 4$, then assume $d\ge 7$. Set $\beta := \lfloor
\binom{m+d-2}{m}/(m+1)\rfloor$.  Assume $3\leq t \le \beta +1$. Let $P$ be a general point of $\tau (X,t)$. Then there
are uniquely determined points $P_1,\dots ,P_{t-2}\in X$ and $Q\in \tau (X)$ such that
$P\in \langle \{P_1,\dots ,P_{t-2}, Q\}\rangle$, i.e. (since $d>2$) there are uniquely determined
points $P_1,\dots ,P_{t-2}\in X$ and a unique tangent vector $\nu$ of $X$ such that $P\in \langle \{P_1,\dots ,P_{t-2}\}\cup  \nu \rangle$.
\end{theorem}

In terms of homogeneous polynomials Theorem \ref{i1} may be rephrased
in the following way.

\begin{theorem}\label{i2}
Fix integers $m \ge 2$ and $d\ge 6$. If $m\le 4$, then assume $d\ge 7$. Set $\beta := \lfloor
\binom{m+d-2}{m}/(m+1)\rfloor$.  Assume $3\leq t \le \beta +1$. Let $P$ be a general point of $\tau (X,t)$ and let $f$ be a homogeneous degree
$d$ form in $\mathbb {K}[x_0,\dots ,x_m]$ associated to $P$. Then $f$ may be written in a unique way
$$f = L_{t-1}^{d-1}L_t + \sum _{i=1}^{t-2} L_i^d$$
with $L_i\in \mathbb {K}[x_0,\dots ,x_m]_1$, $1 \le i \le t$.
\end{theorem}

In the statement of Theorem \ref{i2} the form $f$ is uniquely determined only up to a non-zero scalar, and (as usual in this topic) ``~uniqueness~'' may allow not only a permutation of the forms $L_1,\dots ,L_{t-2}$, but
also a scalar multiplication of each $L_i$.

In terms of symmetric tensors Theorem \ref{i1} may be rephrased
in the following way.

\begin{theorem}\label{iadd}
Fix integers $m \ge 2$ and $d\ge 6$. If $m\le 4$, then assume $d\ge 7$. Set $\beta := \lfloor
\binom{m+d-2}{m}/(m+1)\rfloor$.  Assume $3\leq t \le \beta +1$. Let $P$ be a general point of $\tau (X,t)$ and let $T\in S^d(V_{\mathbb{K}}^*)$ be a symmetric tensor associated to $P$. Then $T$ may be written in a unique way
$$T = v_{t-1}^{\otimes (d-1)}v_t + \sum _{i=1}^{t-2} v_i^{\otimes d}$$
with $v_i\in V^*_{\mathbb{K}}$, $1 \le i \le t$.
\end{theorem}

As above, in the statement of Theorem \ref{iadd} the tensor $T$ and the vectors $v_i$'s are uniquely determined only up to  non-zero scalars.

To prove Theorem \ref{i1}, and hence Theorems \ref{i2} and \ref{iadd}, we adapt the notion and the results on weakly defective varieties
described in \cite{cc}. It is easy to adapt \cite{cc} to joins of different varieties instead of secant varieties of a fixed
variety if a general tangent hyperplane is tangent only at one point (\cite{cc2}). However, a general tangent space of $\tau
(X)$ is tangent to
$\tau (X)$ along a line, not just at the point of tangency. Hence a general hyperplane tangent to $\tau (X,t)$, $t\ge 3$,
is tangent to $\tau (X,t)$ at least along a line. We
prove the following result.

\begin{theorem}\label{i3}
Fix integers $m \ge 2$ and $d\ge 6$. If $m\le 4$, then assume $d\ge 7$. Set $\beta := \lfloor
\binom{m+d-2}{m}/(m+1)\rfloor$.  Assume $t \le \beta +1$.  Let $P$ be a general point of $\tau (X,t)$. Let $P_1,\dots ,P_{t-2}\in X$ and $Q\in \tau (X)$ be the points such that
$P\in \langle \{P_1,\dots ,P_{t-2}, Q\}\rangle$. Let $\nu $ be the tangent vector of $X$ such that $Q$ is a point of $\langle \nu \rangle \setminus \nu _{red}$.
Let
$H\subset \mathbb {P}^N$ be a general hyperplane
containing the tangent space $T_P\tau (X,t)$ of $\tau (X,t)$. Then $H$  is tangent to $X$ only
at the points
$P_1,\dots ,P_{t-2},
\nu _{red}$, the scheme $H\cap X$ has an ordinary node at each $P_i$, and $H$ is tangent to $\tau (X)\setminus X$
only along the line
$\langle \nu \rangle$.
\end{theorem}

\section{Preliminaries}\label{S2}

\begin{notation}
Let $Y$ be an integral quasi-projective variety and $Q\in Y_{reg}$. Let $\{kQ,Y\}$ denote the  $(k-1)$-th infinitesimal neighborhood of $Q$ in $Y$, i.e. the
closed subscheme of $Y$ with $(\mathcal {I}_Q)^k$ as its ideal sheaf. If $Y= \mathbb {P}^m$, then we write $kQ$ instead of $\{kQ,\mathbb {P}^m\}$. The scheme $\{kQ,Y\}$ will be called
a $k$-point of $Y$. We also say that a $2$-point is a double point, that a $3$-point is a triple point and a $4$-point is a quadruple point.
\end{notation}

We give here the definition of a $(2,3)$-point as it is in \cite{cgg}, p. 977.

\begin{definition}\label{23point}
Let $\mathfrak{q}\subset \mathbb{K}[x_0, \ldots , x_m]$ be the reduced ideal of a simple point  $Q\in \mathbb{P}^m$, and let $l\subset \mathbb{K}[x_0, \ldots , x_m]$ be the ideal of a reduced line $L\subset \mathbb{P}^m$ through $Q$. We say that $Z(Q,L)$ is a  $(2,3)$-point if it is the zero-dimensional scheme whose representative ideal is $({\mathfrak{q}}^3 + l^2)$.
\end{definition}

\begin{remark}
Notice that $2Q \subset Z(Q,L) \subset 3Q$. 
\end{remark}

We recall the notion of weak non-defectivity for an integral and non-degenerate projective variety $Y\subset \mathbb {P}^r$ (see \cite{cc}). For any closed subscheme $Z\subset \mathbb {P}^r$ set:
\begin{equation}\label{H}
\mathcal {H}(-Z):= \vert \mathcal {I}_{Z,\mathbb {P}^r}(1)\vert.
\end{equation} 

\begin{notation}\label{contact}
Let $Z\subset \mathbb{P}^r$ be a zero-dimensional scheme
such that $\{2Q,Y\} \subseteq Z$ for all $Q\in Z_{red}$. 
Fix $H\in \mathcal {H}(-Z)$ where $\mathcal {H}(-Z)$ is defined in (\ref{H}). Let $H_c$ be the closure in $Y$
of the set of all $Q\in Y_{reg}$ such that $T_QY \subseteq H$. 
\\ 
The contact locus $H_Z$ of $H$ is the union of all irreducible components of $H_c$ containing at least one point of $Z_{red}$.
\\
We use the notation $H_Z$ only in the case $Z_{red} \subset Y_{reg}$. 
\end{notation}
Fix
an integer
$k
\ge 0$ and assume that $\sigma _{k+1}(Y) $ doesn't fill up the ambient space $ \mathbb {P}^r$. 
Fix a general $(k+1)$-uple of points in $Y$ i.e.
$(P_0,\dots ,P_k)\in Y^{k+1}$ and set 
\begin{equation}\label{Zweak}
Z:= \cup _{i=0}^{k} \{2P_i,Y\}.
\end{equation}
The following definition of weakly $k$-defective variety coincides with the one given in \cite{cc}.
\begin{definition}
A variety $Y\subset \mathbb{P}^r$ is said to be {\it weakly $k$-defective}
if $\dim (H_Z)>0$  for $Z$ as in (\ref{Zweak}). 
\end{definition}

In \cite{cc}, Theorem 1.4, it is proved that if $Y\subset \mathbb{P}^r$ is not weakly $k$-defective, then $H_Z = Z_{red}$ and that
$\mbox{Sing}(Y\cap H) = ( \mbox{Sing}(Y)\cap H)\cup Z_{red}$ for a general $Z=\cup _{i=0}^{k}
\{2P_i,Y\}$ and a general $H\in \mathcal {H}(-Z)$. Notice that
$Y$ is weakly
$0$-defective if and only if its dual variety
$Y^\ast
\subset \mathbb {P}^{r\ast }$ is not a hypersurface. 

In \cite{cc2} the same authors considered also
the case in which $Y$ is not irreducible and hence its joins have as irreducible components the joins of different
varieties.

\begin{lemma}\label{a1}
Fix an integer $y\ge 2$, an integral projective variety $Y$, $L\in \mbox{Pic}(Y)$ and $P\in Y_{reg}$. Set $x:= \dim (Y)$.
Assume $h^0(Y,\mathcal {I}_{(y+1)P}\otimes L) = h^0(Y,L) -\binom{x+y}{x}$. Fix a general $F\in \vert \mathcal
{I}_{yP}\otimes L\vert$. Then $P$ is an isolated singular point of $F$.
\end{lemma}

\begin{proof}
Let $u: Y' \to Y$ denote the blowing-up of $Y$ at $P$ and $E:= u^{-1}(P)$ the exceptional divisor. Since $\dim (Y)=x$, we have $E\cong \mathbb {P}^{x-1}$.
Set $R:= u^\ast (L)$. For each integer $t\ge 0$ we have $u_\ast (R(-tE)) \cong \mathcal {I}_{tP}\otimes L$. Thus the
push-forward
$u_\ast$ induces an isomorphism between the linear system $\vert R(-tE)\vert$ on $Y'$ and the linear system $\vert  \mathcal {I}_{tP}\otimes
L\vert$ on $Y$. Set $M:= R(-yE)$. Since $\mathcal {O}_{Y'}(E) \vert E \cong \mathcal {O}_E(-1)$ (up to the identification of
$E$ with
$\mathbb {P}^{x-1}$), we have $R(-tE)\vert E
\cong \mathcal {O}_E(t)$ for all $t\in \mathbb {N}$. Consider on $Y'$ the exact sequence:
\begin{equation}\label{eqa1}
0 \to M(-E) \to M \to \mathcal {O}_E(y) \to 0
\end{equation}
Our hypothesis implies that $h^0(Y,\mathcal {I}_{yP}\otimes L) = h^0(Y,L) -\binom{x+y-1}{x}$. Thus our assumption implies
$h^0(Y',M(-E)) = h^0(Y',R) -\binom{x+y}{x} = h^0(Y',R)- \binom{x+y-1}{x}-
 \binom{x+y-1}{x-1} = h^0(Y',M) -h^0(E,\mathcal {O}_E(y))$. Thus (\ref{eqa1})
gives the surjectivity of the restriction map $\rho : H^0(Y',M) \to H^0(E,M\vert_{ E})$. Since $y\ge 0$, the line bundle $M\vert
E$ is spanned. Thus the surjectivity of $\rho$ implies that $M$ is spanned at each point of $E$. Hence $M$ is spanned in a
neighborhood of $E$. Bertini's theorem implies that a general $F'\in \vert M\vert$ is smooth in a neighborhood of $E$. Since
$F$ is general and $\vert M\vert \cong \vert \mathcal {I}_{yP}\otimes L\vert$, $P$ is an isolated singular point of
$F$.\end{proof}

\section{$\tau (X,t)$ is not weak defective}\label{S4}

In this section we fix integers $m \ge 2$, $d\ge 3$ and set $N= \binom{m+d}{m}-1$ and $X:= X_{m,d}$. The variety $\tau (X)$ is $0$-weakly defective, because
a general tangent space of $\tau (X)$ is tangent to $\tau (X)$ along a line. Terracini's lemma for joins implies that a general
tangent space of $\tau (X,t)$ is tangent to $\tau (X,t)$ at least along a line (see Remark \ref{e1}). Thus $\tau (X,t)$ is weakly $0$-defective.
To handle this problem and prove Theorem \ref{i1} we introduce another definition, which is tailor-made to this particular
case. As in \cite{cgg} we want to work with zero-dimensional schemes on $X$, not on $\tau (X)$ or $\tau (X,t)$. We
consider 
$X=j_{m,d}(\mathbb{P}^m)$ and the 0-dimensional scheme 
$Z\subset X$ which is the image (via $j_{m,d}$) of the general disjoint union of $t-2$ double points and one $(2,3)$-point of $\mathbb{P}^m$, in the case of \cite{cgg} (see Definition 1). We will often work by identifying $X$ with $\mathbb{P}^m$, so e.g. notice that $\mathcal{H}(-\emptyset)$ is just $|\mathcal{O}_{\mathbb{P}^m}(d)|$.

\begin{remark}\label{e1}
Fix $P\in X$ and $Q\in T_PX\setminus \{P\}$. Any two such pairs $(P,Q)$ are projectively equivalent for
the natural action of $\mbox{Aut}(\mathbb {P}^m)$. We have $Q\in \tau (X)_{reg}$ and $T_Q\tau (X) \supset T_PX$.
Set $D:= \langle \{P,Q\}\rangle$. It is well-known that $D\setminus \{P\}$ is the  set of all $O\in \tau (X)_{reg}$ such
that $T_Q\tau (X) = T_O\tau (X)$ (e.g. use that the set of all $g\in \mbox{Aut}(\mathbb {P}^m)$ fixing $P$ and the line containing
$P$ associated to the tangent vector induced by $Q$
acts transitively on $T_PX\setminus D$).
\end{remark}

\begin{definition}
Fix a general $(O_1,\dots ,O_{t-2},O)\in (\mathbb {P}^m)^{t-1}$ and a general line $L\subset \mathbb {P}^m$
such that $O\in L$. Set $Z:= Z(O,L)\cup \bigcup _{i=1}^{t-2} 2O_i$. We say that the
variety 
$\tau (X,t)$ is not {\it drip defective} if
$\dim (H_Z)=0$ for a general $H\in \vert \mathcal {I}_Z(d)\vert$.
\end{definition}

We are now ready for the following lemma.

\begin{lemma}\label{w1}
Fix an integer $t \ge 3$ such that $(m+1)t < n$. Let $Z_1\subset \mathbb {P}^m$ be a general union of a quadruple point and
$t-2$ double points. Let $Z_2$ be a general union of $2$ triple points and $t-2$ double points. Fix a general disjoint union $Z = Z(O,L)\cup (\cup _{i=1}^{t-2} 2P_i)$ ,  where $Z(O,L)$ is a $(2,3)$-point as in Definition \ref{23point}
and $O$, $L$ and $\{P_1,\dots ,P_{t-2}\} \subset \mathbb {P}^m$ are general. Assume
$h^1(\mathbb {P}^m,\mathcal {I}_{Z_1}(d)) =  h^1(\mathbb {P}^m,\mathcal {I}_{Z_2}(d)) =0$. Then:

\quad (i) $h^1(\mathbb {P}^m,\mathcal {I}_Z(d))=0$; 

\quad (ii) $\tau (X,t)$
is not drip defective;

\quad (iii) a general $H\in \mathcal {H}(-Z)$ has an ordinary quadratic singularity at each $P_i$.
\end{lemma}

\begin{proof}
Set $W:= 3O\cup (\cup _{i=1}^{t-2} 2P_i)$. The definition of a $(2,3)$-point gives that $Z(O,L) \subset 3O$. Thus $Z\subset W
\subset Z_2$. Hence  $h^1(\mathbb {P}^m,\mathcal {I}_Z(d)) \le h^1(\mathbb {P}^m,\mathcal {I}_{Z_2}(d)) =0$. Hence part (i) is proven.\\

To prove part
(ii) of the lemma we need to prove that $\dim (H_Z)=0$ for a general $H\in \mathcal {H}(-Z)$. Since $W \subsetneqq Z_1$ and $h^1(\mathbb {P}^m,\mathcal {I}_{Z_1}(d)) =0$, we have $\mathcal {H}(-W)\ne \emptyset$.
Since $W_{red}=Z_{red}$ and $Z\subset W$, to prove parts (ii) and (iii) of the lemma it is sufficient to prove $\dim ((H_W)_c)=0$ for a general
$H_W\in \mathcal {H}(-W)$, where $W$ is as above and $(H_W)_c$ is as in Notation \ref{contact}. Assume that this is not true, therefore:
\begin{enumerate}
 \item either the contact locus $(H_W)_c$ contains a positive-dimensional component $J_i$ containing some of the $P_i$'s, for $1 \le i
\le t-2$,
\item or the contact locus $(H_W)_c$ contains a positive-dimensional
irreducible component $T$ containing $Q$.
\end{enumerate}
Set $Z_3:=\cup _{i=1}^{t-3} 2P_i$ and $Z':=
3O\cup Z_3$. 

\quad (a) Here we assume the existence of a positive dimensional component $J_i\subset (H_W)_c$ containing one of the $P_i$'s, say for example $J_{t-2}\ni P_{t-2}$.
Thus a general element of $ \vert \mathcal {I}_{W}(d)\vert$ is singular along a positive-dimensional irreducible
algebraic set containing $P_{t-2}$. Let $w: M\to \mathbb {P}^m$ denote the blowing-up of $\mathbb {P}^m$ at the points
$O,P_1,\dots ,P_{t-3}$. Set $E_0:= w^{-1}(O)$ and $E_i:= w^{-1}(P_i)$, $1 \le i \le t-3$. Let $A$ be the only point
of $M$ such that $w(A) =P_{t-2}$. For each integer $y\ge 0$ we have $w_\ast (\mathcal {I}_{yA}\otimes w^\ast (\mathcal
{O}_{\mathbb {P}^m}(d))(-3E_0-2E_1-\cdots -2E_{t-3})) = \mathcal {I}_{Z'\cup yP_{t-2}}(d)$. Applying Lemma
\ref{a1} to the variety $M$, the line bundle $w^\ast (\mathcal
{O}_{\mathbb {P}^m}(d))(-3E_0-2E_1-\cdots -2E_{t-3})$, the point $A$ and the integer $y=2$ we get a contradiction. 

\quad (b) Here we prove the non-existence of a positive-dimensional $T\subset (H_W)_c$ containing $O$. Let $w_1: M_1\to \mathbb
{P}^m$ denote the blowing-up of $\mathbb {P}^m$ at the points
$P_1,\dots ,P_{t-2}$. Set $E_i:= w_1^{-1}(P_i)$, $1 \le i \le t-2$. Let $B\in M_1$ be the only
point of $M_1$ such that $w_1(B)=O$. For each integer $y\ge 0$ we have $w_{1\ast }(\mathcal {I}_{yB}\otimes w_1^\ast (\mathcal
{O}_{\mathbb {P}^m}(d))(-2E_1-\cdots -2E_{t-2})) = \mathcal {I}_{Z'\cup yO}(d)$. Since $h^1(\mathbb {P}^m,\mathcal
{I}_{Z_2}(d)) =0$ and $\vert \mathcal
{I}_{Z_2}(d)\vert \subset \vert \mathcal {I}_Z(d)\vert$, by  Lemma \ref{a1} (with $y=3$) we get a contradiction.
\end{proof} 

In \cite{bb}, Lemmas 5 and 6, we proved the following two lemmas:

\begin{lemma}\label{h1}
Fix integers $m \ge 2$ and $d\ge 5$. If $m\le 4$, then assume $d\ge 6$. Set $\alpha := \lfloor \binom{m+d-1}{m}/(m+1)\rfloor$.
Let $Z_i \subset \mathbb {P}^m$, $i=1,2$, be a general union of $i$ triple points and $\alpha -i$ double points.
Then $h^1(\mathcal {I}_{Z_i}(d))=0$.
\end{lemma}

\begin{lemma}\label{h2}
Fix integers $m \ge 2$ and $d\ge 6$. If $m\le 4$, then assume $d\ge 7$. Set $\beta := \lfloor \binom{m+d-2}{m}/(m+1)\rfloor$.
Let $Z \subset \mathbb {P}^m$ be a general union of one quadruple point and $\beta - 1$ double points.
Then $h^i(\mathcal {I}_Z(d))=0$.
\end{lemma}

We will use the following set-up. 

\begin{notation}\label{e0}
Fix any $Q\in \tau (X)\setminus X$. For $d\geq 3$ the point $Q$ uniquely determines a point $B\in X$ and (up to a non-zero scalar) a tangent vector $\nu$ of $X$ with $\nu _{red} = \{B\}$.
We have $Q\in \langle \nu \rangle \setminus \{B\}$ and $T_Q\tau (X)$ is tangent to $\tau (X)\setminus X$ exactly along the line $\langle \nu \rangle
= \langle \{B,Q\}\rangle$. Let $O\in \mathbb {P}^m$ be the only point such that $j_{n,d}(O) =B$. Let $u_O: \widetilde{X} \to \mathbb {P}^m$ be the blowing-up of $O$. Let $E:= u_O^{-1}(O)$ denote the exceptional divisor.
For all integers $x,e$ set $\mathcal {O}_{\widetilde{X}}(x,eE):= u^\ast (\mathcal {O}_{\mathbb {P}^m}(x))(eE)$. Let $\mathcal {H}$ denote the linear system
$\vert \mathcal {O}_{\widetilde{X}}(d,-3E)\vert$ on $\widetilde{X}$.

\end{notation}

\begin{remark}\label{e1}
When $d \ge 4$, the line bundle $\mathcal {O}_{\widetilde{X}}(d,-3E)$ is very ample, $u_\ast (\mathcal {O}_{\widetilde{X}}(d,-3E))
= \mathcal {I}_{3O}(1)$, $h^0(\widetilde{X},\mathcal {O}_{\widetilde{X}}(d,-3E)) = \binom{m+d}{m} -\binom{m+2}{m}$ and
$h^i(\widetilde{X},\mathcal {O}_{\widetilde{X}}(d,-3E))=0$ for all $i>0$.\end{remark}

\begin{lemma}\label{e2} 
Fix integers $m\ge 2$ and $d \ge 5$.  If $m\le 4$, then assume $d\ge 6$. Set $\alpha := \lfloor \binom{m+d-1}{m}/(m+1)\rfloor$. Fix an integer $t$ such
that $3 \le t \le \alpha$. The linear system $\mathcal {H}$ on $\widetilde{X}$ is not $(t-3)$-weakly defective. For a general $O_1,\dots ,O_{t-2}\in \widetilde{X}$ a general
$H\in \vert \mathcal {H}(-2O_1-\dots -2O_
{t-2})\vert$ is singular only at the points $O_1,\dots ,O_{t-2}$ which are ordinary double points of $H$. 
\end{lemma}

\begin{proof}
Fix general $O_1,\dots ,O_{t-2}\in \widetilde{X}$. Fix $j\in \{1,\dots ,t-2\}$ and set
$Z':= 3O_j\cup \bigcup _{i\ne j} 2O_i$, $Z'':= \cup _{i=1}^{t-2} 2O_i$ and $W:= 3O_j\cup \bigcup _{i\ne j} 2O_i$. We have $u_\ast (\mathcal {I}_{Z'}(d,-3E))
\cong \mathcal {I}_{W\cup 3O}(1)$. The case $i=2$ of Lemma \ref{h1} gives $h^1(\mathcal {I}_Z(d,-3E))=0$. Lemma \ref{a1} applied to a blowing-up of $\mathbb {P}^m$ at $\{O,O_1,\dots ,O_{t-2}\}\setminus \{O_j\}$
shows that a general
$H\in \mathcal {H}(-Z)$ has as an isolated singular point at $O_j$. Since this is true for all $j\in \{1,\dots ,t-2\}$, $\mathcal {H}$ is not $(t-3)$-weakly defective
(just by the definition of weak defectivity). The second assertion follows from the first one and \cite{cc}, Theorem 1.4.
\end{proof}

Now we can apply  Lemmas \ref{w1}, \ref{h1}, \ref{h2} and \ref{e2} and get the following result.

\begin{theorem}\label{w2}
Fix integers $m \ge 2$ and $d\ge 6$. If $m\le 4$, then assume $d\ge 7$. Set $\beta := \lfloor
\binom{m+d-2}{m}/(m+1)\rfloor$. Fix an integer $t$ such that $3 \le t \le \beta +1$. Then $\tau (X,t)$ is not drip defective.
\end{theorem}

\begin{proof}
Fix general $P_1,\dots ,P_{t-2},O\in \mathbb {P}^m$ and a general line $L\subset \mathbb {P}^m$ such that $O\in L$. Set $Z:=Z(O,L)\cup \bigcup _{i=1}^{t-2} 2P_i$,
$W:= 3O\cup \bigcup _{i=1}^{t-2} 2P_{t-2}$, $W':= 3O\cup 3O_1\cup \bigcup _{i=2}^{t-2} 2P_{t-2}$ and $W'':= 4O\cup \bigcup _{i=1}^{t-2} 2P_{t-2}$. Take $O_i\in \widetilde{X}$
such that $u_O(O_i) =P_i$, $1\le i \le t-2$. 
Since $u_{O_\ast }(\mathcal {I}_{2O_1\cup \cdots \cup 2O_{t-2}}(d,-4E)) \cong \mathcal {I}_W(d)$, Lemma \ref{h2} gives $h^1(\mathcal {I}_{2O_1 \cup \cdots \cup 2O_{t-2}}(d,-4E))
=0$. Since $Z(O,L) \subset 3O$, the case $y=3$ of Lemma \ref{a1} applied to the blowing-up of $\mathbb {P}^m$ at $O_1,\dots ,O_{t-2}$ shows that a general
$H\in
\vert
\mathcal {I}_W(d)\vert$ has an isolated singularity at
$O$ with multiplicity at most $3$.
\end{proof}

Recall that $\mbox{Sing}(\tau (X)) = X$ and that for each $Q\in \tau (X)\setminus X$ there is a unique $O\in X$ and a unique tangent vector $\nu$ to $X$ at $O$ such
that $Q\in \langle \nu \rangle$ and that $\langle \nu \rangle \setminus \{O\}$ is the contact locus of the tangent space $T_Q\tau (X)$ with $\tau (X)\setminus X$.

Let $P$ be a general point of $\tau (X,t)$, i.e. fix a general $(P_1,\dots ,P_{t-2},Q) \in X^{t-2}\times \tau (X)$ and a general
$P\in \langle \{P_1,\dots ,P_{t-2},Q\}\rangle$.

\vspace{0.3cm}

\qquad {\emph {Proof of Theorem \ref{i1}.}} Fix a general $P\in \tau (X,t)$, say $P\in \langle \{P_1,\dots ,P_{t-2},Q\}\rangle$ with $(P_1,\dots ,P_{t-2},Q)$ general in $X^{t-2}\times \tau (X)$. Terracini's lemma for joins (\cite{a}, Corollary 1.10)
gives $T_P\tau (X,t) = \langle T_{P_1}X\cup \cdots T_{P_{t-2}}X\cup T_Q\tau (X)\rangle$. Let $O$ be the point of $\mathbb {P}^m$ such that
$Q\in T_{j_{m,d}(O)}X$. Let $\mathcal {H}'$ (resp. $\mathcal {H}''$) be the set of all hyperplane $H\subset \mathbb {P}^N$
containing $T_Q\tau (X)$ (resp. $T_P\tau (X,t)$). We may see $\mathcal {H}'$ and $\mathcal {H}''$ as linear systems on the blowing-up $\widetilde{X}$ of $\mathbb {P}^m$ at $O$.  Take $O_i\in \widetilde{X}$, $1\le i \le t-2$,
such that $P_i=u(O_i)$ for all $i$. 
 We have $\mathcal {H}'' = \mathcal {H}'(-2P_1-\cdots -2P_{t-2})$ and $\mathcal {H} \subseteq \mathcal {H}'$, where $\mathcal {H}$ is defined in Notation \ref{e0}. Since $(P_1,\dots ,P_{t-2})$ is general in $X^{t-2}$ for a fixed $Q$ and $\mathcal {H} \subseteq \mathcal {H}'$, Lemma
\ref{e2} gives that a general $H\in \mathcal {H}''$ intersects $X$ in a divisor which, outside $O$, is singular
only at $P_1,\dots ,P_{t-2}$ and with an ordinary node at each $P_i$. Now assume $P\in \langle \{P'_1,\dots P'_{t-2},Q'\}\rangle$
for some other $(P'_1,\dots ,P'_{t-2},Q')\in X^{t-2}\times \tau (X)$. Since $P$ is general in $\tau (X,t)$ and $\tau (X,t)$ has the expected dimension, the $(t-1)$-ple $(P'_1,\dots ,P'_{t-2},Q')$ is general in $X^{t-2}\times \tau (X)$. Hence $H\cap X$ is singular at each $P'_i$, $1\le i \le t-2$, and with an ordinary node at each $P'_i$. Since
$O$ is not an ordinary node of $H\cap X$, we get $\{P_1,\dots ,P_{t-2}\} = \{P'_1,\dots ,P'_{t-2}\}$. Thus $O=O'$. Hence $H$ is tangent to $\tau (X)_{reg}$ exactly along the line
$\langle \{Q,O\}\rangle \setminus \{O\}$. Hence $Q'\in \langle \{Q,O\}\rangle$. Assume $Q\ne Q'$. Since $P$ is general in $\tau (X,t)$, then $P\notin \tau (X,t-1)$. Hence
$Q'\notin \langle \{P_1,\dots ,P_{t-2}\}\rangle$ and $Q\notin \langle \{P_1,\dots ,P_{t-2}\}\rangle$. Thus $\langle \{P_1,\dots
,P_{t-2},Q\}\rangle \cap \langle \{P_1,\dots ,P_{t-2},Q'\}\rangle = \langle \{P_1,\dots ,P_{t-2}\}\rangle$ if $Q\ne Q'$. Since
$P\in \langle \{P_1,\dots ,P_{t-2},Q\}\rangle \cap \langle
\{P_1,\dots ,P_{t-2},Q'\}\rangle$, we got a contradiction.\qed

\vspace{0.3cm}

\qquad {\emph {Proof of Theorem \ref{i3}.}} The case $t=2$ is well-known and follows from the following fact: for any $O\in X$ and any $Q\in T_OX\setminus \{O\}$ the group $G_O:=
\{g\in \mbox{Aut}(\mathbb {P}^n): g(O)=O\}$ acts on $T_OX$ and the stabilizer $G_{O,Q}$ of $Q$ for this action is the line $\langle \{O,Q\}\rangle$, while
$T_OX\setminus \langle \{O,Q\}\rangle$ is another orbit for $G_{O,Q}$. Thus we may assume $t\ge 3$. Fix a general $P\in \tau
(X,t)$ and a general hyperplane $H\supset T_P\tau (X,t)$. If $H$ is tangent to $\tau (X)$ at a point $Q'\in \tau
(X)\setminus X$, then it is tangent along a line containing $Q'$. Let $E\in X$ be the only point such that $Q'\in T_EX$. We
get $T_EX
\subset
\tau (X,t)$ and that
$H\cap T_EX$ is larger than the double point $2E \subset X$. Theorem \ref{i1} gives that $Q$, $Q'$ and $E$ are
collinear, i.e $H$ is tangent only along the line $\nu$.\qed

\providecommand{\bysame}{\leavevmode\hbox to3em{\hrulefill}\thinspace}

\end{document}